\documentclass[12pt]{article}
\usepackage{amsmath, amssymb, amsfonts, amsthm}
\usepackage[utf8]{inputenc}
\usepackage[T1]{fontenc}
\usepackage{lmodern}
\usepackage{geometry}
\usepackage{hyperref}
\usepackage{xcolor}
\newtheorem{theorem}{Theorem}[section]
\newtheorem{proposition}[theorem]{Proposition}
\newtheorem{lemma}[theorem]{Lemma}

\newtheorem{corollary}[theorem]{Corollary}

\begin{document}

\title{Projections and minimal invariant subspaces in the Hardy space over the bidisk}

\author{
João Marcos R. do Carmo \\
\small Instituto Federal da Bahia - Campus Seabra \\
\small \texttt{joaoribeiro@ifba.edu.br}
\and
Marcos S. Ferreira \\
\small Universidade Estadual de Santa Cruz \\
\small \texttt{msferreira@uesc.br}
}

\date{}

\maketitle

\begin{abstract}
We study the rigidity of the orthogonal projections
$P_{n}:H^2(\mathbb{D}^2)\rightarrow H_{n}$, where
$H^{2}(\mathbb{D}^{2})=\bigoplus_{n=0}^{\infty}H_{n}$ and each
$H_{n}=H^{2}(z)w^{n}$ is a reducing subspace for $T_{z}$, and we
derive properties of the minimal invariant subspaces of $T_{z}^{*}$
. Our main result establishes that the nonzero projections of
a minimal invariant subspace onto the one-variable Hardy space share
a common $t_z^*$-invariant structure. Furthermore, we relate the
results obtained to an approach to the invariant subspace problem.
\end{abstract}



\section{Introduction and background}
Let $\mathcal{H}$ be a Hilbert space and $\mathcal{L}(\mathcal{H})$ the algebra of bounded operators on $\mathcal{H}$. An invariant subspace of $T\in\mathcal{L}(\mathcal{H})$ is a closed subspace $E\subset\mathcal{H}$ such that $TE\subset E$. In this case, we say that $E$ is a $T$-invariant subspace.

The Invariant Subspace Problem (ISP) for complex separable Hilbert space asks if every bounded operator has a nontrivial invariant subspace. Recent works \cite{Chalendar, Cowen, Cowen2} are references for some modern approaches to the ISP.

In 1960, Rota \cite{Rota} introduced a class of operators with an invariant subspace structure so rich enough to model any operator on a Hilbert space. We say that an operator $U\in\mathcal{L}(\mathcal{H})$ is universal for $\mathcal{H}$ if for any $A\in\mathcal{L}(\mathcal{H})$ there exist a $U$-invariant subspace $M$, a non-zero scalar $\lambda$ and a linear isomorphism $X$ of $\mathcal{H}$ onto $M$ such that $UX=\lambda XA$.

The relationship between universal operators and the ISP is as follows.

\begin{theorem}\label{teo1}
If $U$ is universal for a separable, infinite dimensional Hilbert space $\mathcal{H}$ then the ISP is true if and only if every minimal invariant subspace for $U$ is one dimensional. 
\end{theorem}

The main tool thus far for identifying universal operators has been the criterion of Caradus \cite{Caradus}. Caradus showed that, for a separable Hilbert space $\mathcal{H}$, an operator $U\in\mathcal{L}(\mathcal{H})$ is universal if $\operatorname{dim}\operatorname{Ker}(U)=\infty$ and $U$ is surjective. A well-known example of a universal operator on $\mathcal{H}$ is the adjoint of a unilateral shift of infinite multiplicity, which was introduced by Rota. In fact, considering $S^{*}$ acting on
$$
\ell^{2}(\mathcal{H})=\left\{(f_{n})_{n=0}^{\infty}:\sum_{n=0}^{\infty}||f_{n}||^{2}_{\mathcal{H}}<\infty\right\}
$$
by
$$
S^{*}(f_{0},f_{1},f_{2},\cdots)=(f_{1},f_{2},\cdots)
$$
we have that $S^{*}$ satisfies the Caradus criterion and so is an universal operator.

Let $\mathbb{T}^{2}$ be the Cartesian product of 2 copies of $\mathbb{T}=\partial\mathbb{D}$ equiped with the normalized Haar measure $\sigma$. Let $L^{2}(\mathbb{T}^{2})$ denote the usual Lebesgue space and $L^{\infty}(\mathbb{T}^{2})$ the essentially bounded functions with respect to $\sigma$. The Hardy space $H^{2}(\mathbb{D}^{2})$ is the Hilbert space of holomorphic functions $f$ on $\mathbb{D}^{2}$ for which
\begin{equation*}
\|f\|^{2}=\sup_{0<r<1}\int_{\mathbb{T}^{2}}|f(r\zeta)|^{2}d\sigma(\zeta)<\infty.  
\end{equation*}

It is well known that if $f\in H^{2}(\mathbb{D}^{2})$, then the radial limit
$$
f^{*}(\zeta)=\lim_{r\rightarrow1}f(r\zeta)
$$
exists for almost all $\zeta\in\mathbb{T}^{2}$ and
$$
\lim_{r\rightarrow1}\int_{\mathbb{T}^{2}}|f_{r}-f^{*}|^{2}d\sigma=0,
$$
where  $f_{r}(\zeta)=f(r\zeta)$ for all $\zeta\in\mathbb{T}^{2}$. Thus, $H^{2}(\mathbb{D}^{2})$ can be viewed as a bounded subspace of $L^{2}(\mathbb{T}^{2})$. Denote by $H^{\infty}(\mathbb{D}^{2})$ the space of bounded analytic functions on $\mathbb{D}^{2}$. An inner function in $\mathbb{D}^{2}$ is a function $f\in H^{\infty}(\mathbb{D}^{2})$ such that $|f^{*}|=1$ a.e. on $\mathbb{T}^{2}$.

Let $P$ denote the orthogonal projection from $L^{2}(\mathbb{T}^{2})$ onto $H^{2}(\mathbb{D}^{2})$. For a function $\varphi\in L^{\infty}(\mathbb{T}^{2})$, the Toeplitz operator $T_{\varphi}$ with symbol $\varphi$ is defined by
$$
T_{\varphi}f=P(\varphi f)
$$
for $f\in H^{2}(\mathbb{D}^{2})$. Then $T_{\varphi}$ is a bounded linear operator on $H^{2}(\mathbb{D}^{2})$ and its adjoint is given by $T^{*}_{\varphi}=T_{\overline{\varphi}}$. Under analogous conditions, the 1-dimensional Toeplitz operator with symbol $\varphi$ will be denoted by $t_{\varphi}$.

Ferreira and Noor \cite{Ferreira} characterized all analytic Toeplitz operators $T_\varphi$ whose adjoints satisfy the Caradus criterion for universality.

\begin{theorem}\label{t1}
Let $\varphi\in H^{\infty}(\mathbb{D}^{2})$. Then $T^{*}_{\varphi}$ satisfies the Caradus criterion for universality if, and only if, $\varphi$ is invertible in $L^{\infty}(\mathbb{T}^{2})$ but non-invertible in $H^{\infty}(\mathbb{D}^{2})$.
\end{theorem}

In particular the backward shift operators $T^{*}_{z}$ and $T^{*}_{w}$ are universal. Thus by Theorem \ref{teo1} it follows that the ISP has a positive solution if and only if every minimal $T_{z}^{*}$-invariant subspace in $H^{2}(\mathbb{D}^{2})$ is one dimensional, that is, is generated by an eigenvector. The description off all $T^{*}_{z}$-invariant subspaces is an important problem in multivariable operator theory (see \cite{Sarkar, Yang}) which could therefore lead to a solution of the ISP.

Let $H^{2}(z)$ and $H^{2}(w)$ denote the classical Hardy spaces over $\mathbb{D}$ in the variables $z$ and $w$ respectively. Then given $g\in H^{2}(\mathbb{D}^{2})$, we write
$$
g(z,w)=\sum_{n=0}^{\infty}g_{n}(z)w^{n}
$$
for the Taylor series expansion of $g$ with respect to $w$-variable where
$$
g_{n}(z)=\int_{\mathbb{T}}g(z,w)\overline{w^{n}}dw
$$
for each $n$ and
$$
\sum_{n=0}^{\infty}\|g_{n}\|^{2}_{H^{2}(z)}<\infty.
$$

Thus if we denote $H_{n}=H^{2}(z)w^{n}$ for each $n\in\mathbb{N}$, we have the decomposition
\begin{equation}\label{eq2}
H^{2}(\mathbb{D}^{2})=\bigoplus_{n=0}^{\infty}H_{n}
\end{equation}
and each $H_{n}$ is a reducing subspace for $T_{z}$. Let $P_{n}:H^2(\mathbb{D}^2)\rightarrow H_{n}$ be the orthogonal projection onto $H_{n}$. In \cite{Chen}, the authors used this decomposition of $H^{2}(\mathbb{D}^{2})$ to study invariant subspaces of the kernel of $T^{*}_{z}$. An important fact that will be useful in our work is the following.

\begin{proposition}
Let $M\subset H^{2}(\mathbb{D}^{2})$ be a $T^{*}_{z}$-invariant subspace. Then, for every $n\in\mathbb{N}$, $\overline{P_{n}(M)}$ is $T^{*}_{z}$-invariant subspace. Equivalently, under the identification $H_{n}=H^{2}(\mathbb{D})w^{n}$, $\overline{P_{n}(M)}$ corresponds to a $t_{z}^{*}$-invariant subspace of $H^{2}(\mathbb{D})$.
\end{proposition}
\begin{proof}
Since $H_n$ is a reducing subspace for $T_z$, the orthogonal projection $P_n$ commutes with $T_z^*$. Hence, for $f=P_{n}g\in P_n(M)$, with \(g\in M\),
$$ 
T_{z}^{*}f=T_{z}^{*}P_{n}g=P_{n}T_{z}^{*}g\in P_{n}(M)
$$
because $T_{z}^{*}g\in M$. Thus $P_{n}(M)$ is $T_{z}^{*}$-invariant. On the other hand, since $H_{n}=H^{2}(\mathbb{D})w^{n}$ we have
$$
T_{z}^{*}(fw^{n})=(t^{*}_{z}f)w^{n}
$$
for all $f\in H^{2}(\mathbb{D})$. So considering
$$
V_{n}=\overline{\{f\in H^{2}(\mathbb{D}):fw^{n}\in P_{n}(M\}}
$$
we obtain $t^{*}_{z}V_{n}\subset V_{n}$ whence $V_{n}$ is $t^{*}_{z}$-invariant and $\overline{P_{n}(M)}=V_{n}w^{n}$.
\end{proof}

Given $\varphi\in H^{\infty}(\mathbb{D})$ and $g\in H^{2}(\mathbb{D})$, the minimal invariant subspace for $T^{*}_{\varphi}$ that contains $g$ will be denoted by
$$
V_{T_{\varphi}^{*},g}=\overline{\operatorname{span}\{(T^{*}_{\varphi})^{n}g:n\in\mathbb{N}\}}.
$$

Let $g\in M$. In \cite{Carmo}, Do Carmo and Ferreira considered the operator $J_{g,\varphi}:\ell^{2}(\mathbb{C})\rightarrow\overline{P_{n}(M)}$ given by 
$$
J_{g,\varphi}(\beta)=\sum_{i=0}^{\infty}\beta_i((t_{\varphi}^{*})^ig_n)w^n
$$
where $\beta=(\beta_i)_{i=0}^{\infty}$ and proved the following conditions sufficient for the ISP to be true.

\begin{theorem}\label{t2}
If there exists an inner function $\varphi\in H^2(\mathbb{D})$ such that for each  $g\in H^2(\mathbb{D})$, there is an invariant subspace $U\subset V_{t^{*}_{\varphi},g}$ of $t_{\varphi}^{*}$ so that $U\cap J_{g,\varphi}(l^2(\mathbb{C}))\neq \lbrace 0\rbrace$, then the ISP is true.
\end{theorem}

The decomposition of $H^{2}(\mathbb{D}^{2})$ given in \eqref{eq2} plays a central role in Theorem \ref{t2}. It allows us to study a $T_z^*$-invariant subspace through its projections onto the slices $H_n$. In particular, the operators $P_n$ provide a natural way to transfer the analysis of invariant subspaces in $H^2(\mathbb D^2)$ to the one-variable Hardy space $H^2(\mathbb D)$. This slice structure is the main tool we will use to investigate minimal $T_z^*$-invariant subspaces and the relations among their projections.

\section{Main Results}

Consider the following classes of invariant subspaces
$$
\mathcal{M}=\{M\subset H^2(\mathbb{D}^2): M \ \hbox{is a minimal invariant subspace for} \ T^{*}_{z}\}
$$
and
$$
\mathcal{M}_{0}=\{M\in\mathcal{M}:\operatorname{dim}(M)=1\}.
$$

Since $T^{*}_{z}$ is an universal operator for $H^2(\mathbb{D}^2)$ if we have $\mathcal{M}\subset\mathcal{M}_{0}$, then the ISP is true. Note that to show that $\mathcal{M}\subset\mathcal{M}_{0}$ it is enough to show that if $M$ is an $T^{*}_{z}$-invariant subspace such that $\operatorname{dim}(M)\neq1$ then $M$ is non minimal.

For all $f\in H^2(\mathbb{D})$, we consider $f^{\sharp}\in H^2(\mathbb{D})$ given by
$$
f^{\sharp}(z)=\overline{f(\overline{z})}.
$$

In our first result (Proposition \ref{a2}), we show that for a minimal invariant subspace $M$ of $T_{z}^{*}$, two non-trivial projections $\overline{P_n(M)}$ and $\overline{P_m(M)}$ cannot be distinct. To this end, we need to introduce the operator $L_g$ and some of its consequences.

\begin{lemma} \label{l31}
Let $g\in H^2(\mathbb{D})$. The operator $L_{g}:H^{2}(\mathbb{D})\rightarrow H^{2}(\mathbb{D})$ given by 
$$ L_g(f)=\sum_{n=0}^{\infty}\langle t_z^{*n}g,f^{\sharp}\rangle z^n $$
is well defined.
\end{lemma}
\begin{proof}
Let $g(z)=\displaystyle\sum_{i=0}^{\infty}\alpha_i z^i$. Since $(t_z^{*})^kg(z)=\displaystyle\sum_{i=0}^{\infty} \alpha_{i+k} z^i$, we have $\langle (t_z^{*})^kg,z^l\rangle=\alpha_{l+k}$ and so
$$
\langle \sum_{i=0}^{\infty}\beta_i(t_z^{*})^ig,z^l\rangle=\sum_{i=0}^{\infty} \beta_i\alpha_{l+i}.
$$

Moreover for $f(z)=\displaystyle\sum_{i=0}^{\infty}\beta_{i}z^i$ it follows that
$$
\langle (t_z^{*})^lg,f^\sharp\rangle=\langle\sum_{j=0}^{\infty} \alpha_{j+l} z^j,\sum_{i=0}^{\infty}\overline{\beta_{i}}z^i\rangle=\sum_{i=0}^{\infty} \beta_i\alpha_{l+i}
$$
therefore
$$
\langle \sum_{i=0}^{\infty}\beta_i(t_z^{*})^ig,z^l\rangle=\langle (t_z^{*})^lg,f^\sharp\rangle
$$
and so $L_{g}f(z)=\displaystyle\sum_{i=0}^{\infty}\beta_i(t_z^{*})^ig$ it is well defined since $(\beta_0,\beta_1,\ldots)\mapsto\displaystyle\sum_{i=0}^{\infty}\beta_i(t_z^{*})^ig$ is a bounded operator of $\ell{^2}(\mathbb{C})$ onto $H^2(\mathbb{D}^2)$.
\end{proof}

The operator $L_g$ introduced in previous Lemma can be viewed as the one-variable counterpart of the operator $J_{g,\varphi}$ (see Theorem \ref{t2}). Indeed, taking \(\varphi(z)=z\) and identifying the slice $H_n=H^2(\mathbb D)w^n$ with $H^2(\mathbb D)$, the restriction of \(J_{g,\varphi}\) to the \(n\)-th slice is given by
$$
J_{g,z}(\beta) = \sum_{i=0}^{\infty}\beta_i(t_z^*)^i(g_n)w^n, 
$$
whereas, if $f(z)=\displaystyle\sum_{i=0}^{\infty}\beta_i z^i$, Lemma \ref{l31} gives
$$ 
L_{g_n}(f) = \sum_{i=0}^{\infty}\beta_i(t_z^*)^i(g_n). 
$$

Thus $J_{g,z}(\beta)=L_{g_n}(f)w^n$. Consequently, under the natural identification $H_n\simeq H^2(\mathbb D)$, the operators $J_{g,z}$ and $L_{g_n}$ coincide.

Note that the Theorem \ref{t2} uses the range of $J_{g,\varphi}$ to detect a smaller invariant structure inside a slice of a $T_\varphi^*$-invariant subspace. In the present setting, $L_g$ provides precisely this mechanism for the backward shift $t_z^*$: its range is generated by the orbit $\{(t_z^*)^i g:i\geq0\}$.

In particular, if $M\subset H^2(\mathbb{D}^2)$ is an invariant subspace of $T^{*}_{z}$ and $g\in P_n(M)$ we have
$$
\sum_{i=0}^{\infty}\beta_i(t_z^{*})^ig\in P_n(M)
$$
for all $\beta=(\beta_0,\beta_1,...)\in l^2(\mathbb{C})$. Thus we have $L_{g}(H^2(\mathbb{D}))\subset P_n(M)$.

Hence, the operator $L_g$ allows us to transfer information about the cyclic subspace generated by $g$ under $t_z^*$ to the corresponding projection of the $T_z^*$-invariant subspace $M$.

The kernel of the operator $L_g$ is well described
\begin{equation}\label{eq1}
\operatorname{Ker}(L_g)=\{f\in H^2(\mathbb{D}):f^\sharp\in \overline{\operatorname{span}\{(t_{z}^{*})^{n}g:n\in\mathbb{N}\}}^\perp\}
\end{equation}
what is
$$
\operatorname{Ker}(L_g)=\{f\in H^2(\mathbb{D}):f^\sharp\in V_{t^{*}_{z},g}^{\perp}\}.
$$

The operator $L_g$ provides a useful link between the orbit of a vector in a slice and the corresponding projection of an invariant subspace. We shall use this connection to establish the following rigidity property of minimal invariant subspaces.

\begin{proposition}\label{a2}
Let $M\subset H^2(\mathbb{D}^2)$ an invariant subspace of $T^{*}_{z}$. If there are $n,m\in \mathbb{N}$ such that $\overline{P_n(M)}\neq \overline{P_m(M)}$ where both are non-trivial, then either $M\in\mathcal{M}_{0}$ or $M\notin\mathcal{M}$.
\end{proposition}
\begin{proof}
Since $\overline{P_n(M)}\neq \overline{P_m(M)}$ we can choose $0\neq h\in H^2(\mathbb{D})$ with $h^{\sharp}\in \overline{P_m(M)}^{\perp}$ but $h^{\sharp}\notin\overline{P_n(M)}^{\perp}$.

Once $\overline{P_n(M)}$ and $\overline{P_m(M)}$ are non null, for each $0\neq g\in P_n(M)$, there are $u\in M$, with $P_n(u)=g$ and $P_m(u)\neq0$ or $M\notin\mathcal{M}$. 

We assume that $\overline{P_n(M)}=\overline{\operatorname{span}\{(t^{*}_z)^{l}g:l\in\mathbb{N}\}}$. Because 
$$
h^\sharp\notin  \overline{P_n(M)}^\perp=\overline{\operatorname{span}\{ (t^{*}_z)^{l}g:l\in\mathbb{N}\}}^\perp
$$
it follows from \eqref{eq1} that $h\notin\operatorname{Ker}(L_g)$ and therefore $L_g(h)\neq 0$. 

On the other hand, we have that $h^{\sharp}\in \overline{P_m(M)}^{\perp}$ and since $P_{m}(M)$ is $t^{*}_{z}$-invariant, that is
$$
(t^{*}_{z})^{l}P_{m}(u)\in P_{m}(M)
$$
it follows that
$$
\langle(t^{*}_{z})^{l}P_{m}(u),h^{\sharp}\rangle=0
$$
and again from \eqref{eq1} we have $L_{P_m(u)}(h)=0$.

Now for $h(z)=\displaystyle\sum_{i=0}^{\infty}h_{i}z^{i}$ considering 
$$
f=\sum_{i=0}^{\infty}h_i(T_{z}^{*})^{i}u=L_{u}(h)
$$
once $P_{l}$ commutes with $T^{*}_{z}$ it follows that
$$
P_{l}(f)=P_{l}\left(\sum_{i=0}^{\infty}h_i(T_{z}^{*})^{i}u\right)=\sum_{i=0}^{\infty}h_{i}(T^{*}_{z})^{i}P_{l}(u)
$$
and so $P_{l}f=L_{P_l(u)}(h)$ by Lemma \ref{l31}. Therefore for $l=n$, then $P_n(f)=L_g(h)\neq 0$ and for $l=m$, then $P_m(f)=L_{P_m(u)}h=0$ hence $f\neq 0$ and so $M\notin \mathcal{M}$.
\end{proof}

The previous proposition yields a rigidity property for the projections of minimal invariant subspaces. In particular, if $M$ is an infinite-dimensional minimal $T_z^*$-invariant subspace, then all nontrivial projections $P_n(M)$ coincide with a single $t_z^*$-invariant subspace of $H^2(\mathbb D)$.

\begin{corollary}
If $M\in\mathcal{M}$ and $\operatorname{dim}(M)=\infty$, then there exists an infinite dimensional $t^{*}_{z}$-invariant subspace $V\subset H^{2}(\mathbb{D})$ such that, for every $n\in \mathbb{N}$, either 
$$
P_n(M)=\lbrace 0\rbrace \ \text{or} \ \overline{P_n(M)}=V.
$$
\end{corollary}
\begin{proof}
Let $0\neq g\in M$. By decomposition \eqref{eq1}, there exists $n_{0}\in\mathbb{N}$ such that $P_{n_{0}}(g)\neq0$ and so $P_{n_{0}}(M)\neq\{0\}$. Since $M\in\mathcal{M}$ and $\operatorname{dim}(M)=\infty$, Proposition \ref{a2} implies that for every $n\in\mathbb{N}$
$$
P_{n}(M)=\lbrace 0\rbrace \ \text{or} \ \overline{P_{n}(M)}=\overline{P_{n_{0}}(M)}.
$$

Define $V=\overline{P_{n_{0}}(M)}$. Since $P_{n_{0}}(M)$ is $t^{*}_{z}$-invariant, $V$ is as well. Moreover $V$ is infinite dimensional. 

Thus either $P_{n}(M)=\{0\}$ or $\overline{P_{n}(M)}=\overline{P_{n_{0}}(M)}=V$, as desired.
\end{proof}

For each non-zero $g\in H^{2}(w)$, let $P_{g}:H^2(\mathbb{D}^2)\rightarrow H^{2}(z)$ be the operator defined by
$$
P_{g}(f)(z)=\int_{\mathbb{T}}f(z,w)\overline{g}(w)dw=\langle f,g\rangle_{w}(z)
$$
for $f\in H^2(\mathbb{D}^2)$, where the subscript under the inner product denotes the variable of integration.

The operator $P_g$ allows us to combine the information contained in the coordinate projections $P_i$. The following lemma makes this relationship explicit and will be useful in comparing projections associated with different vectors in $H^2(w)$.

\begin{lemma}
Let $g(w)=\displaystyle\sum_{i=0}^{\infty}a_iw^i\in H^2(w)$. For each $f(z,w)\in H^2(\mathbb{D}^2)$ we have
$$
P_g(f)(z)=\sum_{i=0}^{\infty}\overline{a_i}P_i(f)(z).
$$
\end{lemma}
\begin{proof}
If $\displaystyle f(z,w)=\sum_{n=0}^{\infty}f_n(z)w^n$, we have 
$$ \begin{array}{rcl}
P_{w^i}(f)(z)&=&\displaystyle\int_{\mathbb{T}}f(z,w)\overline{w^i}dw  \\
&=&\displaystyle\int_{\mathbb{T}}\sum_{n=0}^{\infty}f_n(z)w^n\overline{w^i}dw\\
             &=&f_{i}(z)\\
             &=&P_{i}f(z,w).
    \end{array}
$$

On the other hand, for $\displaystyle g(w)=\sum_{i=0}^{\infty}a_{i}w^{i}$ it follows that
$$
\begin{array}{rcl}
P_{g}(f)(z)&=&\displaystyle\int_{\mathbb{T}}f(z,w)\overline{g}(w)dw  \\
&=&\displaystyle\int_{\mathbb{T}}\sum_{n=0}^{\infty} f_n(z)w^n\left(\sum_{i=0}^{\infty}\overline{a_{i}w^{i}}\right)dw\\
&=&\displaystyle\int_{\mathbb{T}}\sum_{n=0}^{\infty} \sum_{i=0}^{\infty}f_n(z)w^n\overline{a_i}\overline{w^i}dw\\
&=&\displaystyle\sum_{i=0}^{\infty} \overline{a_i}f_{i}(z)\\
&=&\displaystyle\sum_{i=0}^{\infty} \overline{a_i}P_{i}f(z,w)
\end{array}
$$
as desired.
\end{proof}

The preceding lemma shows that the operator $P_g$ can be expressed in terms of the coordinate projections $P_i$. We next use this fact to compare projections associated with orthogonal vectors in $H^2(w)$ and derive a rigidity property for minimal invariant subspaces.

\begin{lemma}\label{a3}
Let $g, h\in H^2(w)$ be orthogonal and let $M\subset H^2(\mathbb{D}^2)$ be a $T^{*}_{z}$-invariant subspace. If both $P_g(M)$ and $P_h(M)$ are nonzero and  $\overline{P_g(M)}\neq \overline{P_h(M)}$, then either $\dim(M)=1$ or $M\notin\mathcal{M}$.
\end{lemma}
\begin{proof}
Since $g\perp h$ there exists an orthonormal basis $\lbrace u_0,u_1,...\rbrace$ of $H^2(w)$ such that
$$
u_0=\frac{g}{\| g\|} \ \mbox{and} \ u_1=\frac{h}{\| h\|}.
$$

For $x\in H^2(\mathbb{D}^2)$, write $\displaystyle x=\sum_{i=0}^{\infty} P_{u_i}(x)u_i$ and define the unitary operator
$$
\Phi(x)=\sum_{i=0}^{\infty} P_{u_i}(x)w^i.
$$

Since $P_{u_{i}}$ commutes with $T^{*}_{z}$, we have
$$
\Phi (T_{z}^{*}x)=\sum_{i=0}^{\infty}P_{u_{i}}(T^{*}_{z}x)w^{i}=\sum_{i=0}^{\infty}t^{*}_{z}P_{u_{i}}(x)w^{i}=T_{z}^{*}\Phi(x)
$$
what is $\Phi T_{z}^{*}=T_{z}^{*}\Phi$. In particular, $\Phi(M)$ is a $T^{*}_{z}$-invariant subspace and
$$
M\in\mathcal{M}\iff\Phi(M)\in\mathcal{M} \ \text{and} \ \operatorname{dim}\Phi(M)=\operatorname{dim}M
$$
and the result follows from Proposition \ref{a2}.
\end{proof}

The previous lemma shows that, for a minimal invariant subspace, two distinct nontrivial projections associated with orthogonal vectors cannot occur. The following corollary extends this rigidity to multiple orthogonal directions.

\begin{corollary}\label{a1}
Let $f,g,h\in H^2(w)$ satisfy $f\perp h$ and $g\perp h$. If $M\in\mathcal{M}$ and $P_f(M)$, $P_g(M)$ and $P_h(M)$ are nonzero, then $\overline{P_f(M)}=\overline{P_g(M)}=\overline{P_h(M)}$.
\end{corollary}
\begin{proof}
Since $f\perp h$ and both $P_f(M)$ and $P_h(M)$ are nonzero, Lemma \ref{a3} implies that $\overline{P_f(M)}=\overline{P_h(M)}$ for a minimal invariant subspace $M$. Simirlarly $\overline{P_g(M)}=\overline{P_h(M)}$.
\end{proof}

The preceding results establish a strong rigidity property for the projections of minimal $T_z^*$-invariant subspaces. We now extend this property from the coordinate projections $P_n$ to the more general operators $P_g$, associated with arbitrary vectors $g\in H^2(w)$. The following theorem to shows that all nonzero projections of a minimal invariant subspace determine the same $t_z^*$-invariant subspace of $H^2(\mathbb D)$.

\begin{theorem}\label{t3}
If $M\in\mathcal{M}$, then for every $g\in H^2(w)$ there exists a $t_z^*$-invariant subspace $V\subset H^2(\mathbb{D})$ such that either $P_g(M)=\{0\}$ or $\overline{P_g(M)}=V$.
\end{theorem}
\begin{proof}
Fix $g\in H^2(w)$. If $P_g(M)=\{0\}$ there is nothing to prove. Hence assume $P_g(M)\neq \lbrace 0\rbrace$. We divide the proof into two cases.

\textbf{Case 1}: $P_n(M)=\{0\}$ for some $n\in \mathbb{N}$.
Choose the set
$$
I=\{k\in\mathbb{N}:P_{k}(M)\neq\{0\}\}.
$$

Since $P_n(M)=\{0\}$ we have $I\neq\mathbb{N}$. If $g\perp w^{k}$ for some $k\in I$, then by Corollary \ref{a1} we have
$$
\overline{P_{g}(M)}=\overline{P_{w^{k}}(M)}=\overline{P_{k}(M)}
$$
and thus the conclusion follows with $V=\overline{P_{k}(M)}$.

Suppose now that $\langle g,w^{k}\rangle\neq0$ for all $k\in I$. Choose $n\notin I$. Then $P_{n}(M)=\{0\}$ and therefore, for every $\lambda\in\mathbb{C}$, it follows that
$$
P_{g+\lambda w^{n}}(M)=P_{g}(M).
$$

For $k\in I$, since both
$$
\langle g,w^{k}\rangle\neq0 \ \text{and} \ \langle g,w^{n}\rangle\neq0
$$
there exists nonzero $\lambda_{1},\lambda_{2}\in\mathbb{C}$ such that $g\perp q$ where $q:=\lambda_{1}w^{k}+\lambda_{2}w^{n}$. Moreover $\overline{P_{q}(M)}=\overline{P_{k}(M)}$ because $P_{n}(M)=\{0\}$. Thus, since $P_{q}(M)$ and $P_{g}(M)$ are nonzero and $q\perp g$, Corollary \ref{a1} gives
$$
V:=\overline{P_{k}(M)}=\overline{P_{q}(M)}=\overline{P_{g}(M)}
$$
as desired.

\textbf{Case 2}: $P_n(M)\neq \lbrace 0\rbrace$ for all $n\in \mathbb{N}$.
Since $M\in\mathcal{M}$ and $P_g(M)\neq \lbrace 0\rbrace$, choose $0\neq h\in M$ such $M=V_{h,T^{*}_{z}}$. Write
$$
h(z,w)=\sum_{n=0}^{\infty}h_{n}w^{n}, \ \text{where} \ h_{n}=P_{n}(h).
$$

By the minimality of $M$, there exists $n_{1}, n_{2}\in\mathbb{N}$ such that $h_{n_{1}}\notin\operatorname{span}\{{h_{n_{2}}}\}$. Again, since
$$
\langle g,w^{n_{1}}\rangle\neq0 \ \text{and} \ \langle g,w^{n_{2}}\rangle\neq0
$$
we can choose nonzero $\lambda_{1}, \lambda_{2}\in\mathbb{C}$ such that $g\perp q$ with $q:=\lambda_{1}w^{n_{1}}+\lambda_{2}w^{n_{2}}$. It then follows from Corollary \ref{a1} that $\overline{P_{g}(M)}=\overline{P_{q}(M)}$ because $P_{q}(M)\neq\{0\}$.

Now if $n\neq n_{1},n_{2}$, then $q\perp w^{n}$ and therefore $\overline{P_{q}(M)}=\overline{P_{n}(M)}$. Consequently
$$
\overline{P_{g}(M)}=\overline{P_{n}(M)}
$$
for every $n\neq n_{1},n_{2}$. Thus taking $n_{0}\notin\{n_{1},n_{2}\}$, we obtain
$$
V:=\overline{P_{n_{0}}(M)}=\overline{P_{g}(M)}
$$
which is $t^{*}_{z}$-invariant.
\end{proof}

Theorem \ref{t3} to shows that although a minimal $T_z^*$-invariant subspace $M$ of $H^2(\mathbb D^2)$ may have nontrivial projections in several directions of the $w$-variable, all its nonzero projections onto $H^2(\mathbb D)$ are governed by a single $t_z^*$-invariant subspace $V$. Thus, the apparent dependence on the choice of the projection disappears at the level of minimal invariant subspaces.

The decomposition \eqref{eq2} to reduce information about invariant subspaces of $T_z^*$ to one-variable invariant subspaces of $t_z^*$, while Theorem \ref{t3} shows that these one-variable shadows are not independent: whenever they are nontrivial, they coincide with the same invariant subspace $V$. In this sense, the theorem establishes a rigid structure for the projections of minimal invariant subspaces. Moreover this structure is closely related to the mechanism in Theorem \ref{t2}. The operators $P_g$ and $L_g$ provide a link between $T_z^*$-invariant subspaces of $H^2(\mathbb D^2)$ and cyclic $t_z^*$-invariant subspaces of $H^2(\mathbb D)$. Theorem \ref{t3} shows that this correspondence is highly constrained for minimal invariant subspaces, and therefore provides a useful structural step toward the study of the ISP via universal operators.



\end{document}